\documentclass[12pt,leqno]{amsart}
\usepackage{amssymb}
\usepackage{amsmath,amssymb,color}
\usepackage{geometry}

\allowdisplaybreaks

\newtheorem{thm}{Theorem}
\newtheorem{lem}{Lemma}[section]

\newtheorem{rmk}{Remark}[section]

\numberwithin{equation}{section}

\newcommand{\va}{\varphi}
\newcommand{\NZ}{\mathbb{N} \cup \{0\}}
\newcommand{\N}{\mathbb{N}}

\newcommand{\ep}{\varepsilon}

\title[Coefficient Inverse Problem in Fractional Diffusion Equation]
{Simultaneous uniqueness for a coefficient inverse problem in
one-dimensional fractional diffusion equation from an interior point
measurement}

\author[Jing, Li and Yamamoto]{Xiaohua Jing$^1$,\ Zhiyuan Li$^{2}$ and Masahiro Yamamoto$^{3,4}$}

\thanks{
\hspace{-1.4em}$^1\;$School of Sciences, Chang'an University, Xi'an, 710064,
Shaanxi, China E-mail: {\tt xhjing@chd.edu.cn}\\
$^2\;$School of Mathematics and Statistics, Ningbo University, Ningbo 315211,
Zhejiang,
China. E-mail: {\tt lizhiyuan@nbu.edu.cn}\\
$^3\;$Graduate School of Mathematical Sciences, The University
of Tokyo, Komaba, Meguro, Tokyo 153-8914, Japan, \\
$^4$Department of Mathematics, Faculty of Science, Zonguldak B\"ulent Ecevit
University,
Zonguldak 67100, Turkiye.
E-mail: myama@ms.u-tokyo.ac.jp }

\begin{document}
\maketitle

\begin{abstract}
This article is concerned with an inverse problem of simultaneously
determining a spatially varying coefficient and a Robin coefficient
for a one-dimensional fractional diffusion equation with a time-fractional
derivative of order $\alpha\in(0,1)$.  We prove the
uniqueness for the inverse problem by observation data at one interior point
over a finite time interval, provided that a coefficient
is known on a subinterval.
Our proof is based on the uniqueness in the inverse spectracl problem
for a Sturm-Liouville problem by means of the Weyl $m$-function and
the spectral representation of the solution to
an initial-boundary value problem for the fractional diffusion equation.
\end{abstract}




\section{Introduction}
Due to the non-locality of the fractional derivative, the fractional diffusion
equation performs well in modeling anomalous phenomena which are characterized
by memory inheritance and path dependence. Fractional diffusion equations have
attracted great attention because of their widespread practical applications
in many
aspects, such as physics \cite{Metzler2000random},
ecology \cite{Brockmann2006scaling}, geological exploration
\cite{Zhou2015adaptive}, complex viscoelastic materials
\cite{Giona1992fractional} and so on.

Let $T>0$ and $h,H \ge 0$.
We consider the following initial-boundary value problem:
\begin{equation}\label{eq-gov}
\left\{
\begin{alignedat}{2}
& d_{t}^{\alpha}u(x,t) = \partial_{x}^{2}u(x,t) + q(x)u(x,t) &\quad&
\mbox{in }(0,1)\times(0,T), \\
& \partial_x u(0,t) - hu(0,t) = 0 &\quad&   \mbox{in } (0,T), \\
& \partial_x u(0,t) + Hu(1,t) = \eta(t) &\quad&   \mbox{in } (0,T), \\
& u(x,0) = 0 &\quad& \mbox{in }(0,1),
\end{alignedat}
\right.
\end{equation}
where $\alpha\in(0,1)$ is a parameter specifying the
large-time behavior of the waiting-time distribution function
(e.g. \cite{Podlubny1998fractional}), and $d_{t}^{\alpha}u$ denotes
the Caputo derivative of order $\alpha\in(0,1)$, which can be defined for
an absolutely continuous function $u\in AC[0, T]$:
$$
d_{t}^{\alpha}u(t)=\frac{1}{\Gamma(1-\alpha)}\int ^{t}_{0}
(t-s)^{-\alpha}u'(s)ds, \quad 0<t<T.
$$
The coefficient $q(x)$ is called a potential, and describes
a spatially dependent source or sink term in a bar, while
Robin coefficients $h, H$ are important physical parameters governing
the convection between the solute in a body and one in the ambient environment
\cite{WRY}.
Thus the potential and the Robin coefficients characterize thermal properties
of the conductive materials in the interior and on the boundary points, and
are of significant practical interest in thermal engineering
e.g., corrosion \cite{JZ09}, and nondestructive evaluation \cite{Beck85}.

Once $\alpha$, $q$, $h$, and $H$ are known, solving the initial-boundary value
problem with given parameters is called a forward problem.
However, several parameters in practical
applications among initial data, a coefficient and a source term are often
unknown and
cannot be measured directly in the model, and these unknown quantities need
to be determined by available data of the solution to the forward problem,
and this is an inverse problem.

Many works have been done on the mathematical analysis of inverse problems
for time-fractional diffusion equations.
As for inverse source for time-fractional equations, there are many
works and here we are restricted to
\cite{Jin2012inverse,liu2016strong,wei2016inverse, zhang2011inverse,ZL21},
\cite{LY}, \cite{THNZ}, \cite{Cheng2009uniqueness, KLLY20, Li2013simultaneous,
MR3465303}.

In particular, for the inverse problems of determining potentials of
fractional diffusion equations with non-homogeneous boundary conditions,
there are few articles and we refer to
\cite{Jin2021recovering, Jing2020simultaneous, Rundell2018recovery,
Rundell2021uniqueness}.

It reveals that related to non-homogeneous boundary conditions,
the above mentioned works consider only boundary measurements.
However, in practical cases such as pollution problems,
the spatial domain is vast and the access to the boundary is not
realistic, so that boundary observation data are difficult and interior
data are more available.
This is the main motivation for the current article.

As for the inverse problems for one-dimensional diffusion equations
(i.e., $\alpha=1$) and the Sturm-Liouville problems by
the interior data at one point, we can refer to  \cite{pierce1979unique}.

In general, the uniqueness by interior data is much more difficult
than the boundary data, and we can prove the uniqueness for potentials
only in a subinterval in $(0,1)$.
Therefore, we will discuss the uniqueness for our
inverse problem under assumption that potentials under consideration
are known in a fixed subinterval.

Thus we formuate our main subject as
\\
{\bf Inverse problem by interior data}
{\it
Let $T>0$, $d\in (0,1)$.
Let $H$ and $\eta(t)$, $0<t<T$, be given and let $x_0\in (0,1)$ be
a fixed point.
We assume that $q$ is known in a subinterval $(d,1)\subset(0,1)$.
Then, we will discuss the simultaneous determination of $q(x)$, $x\in[0,1]$
and $h$ from observation data $u(x_0,\cdot)$ in $(0,T)$.
}

\section{Main results and outline}

In this article, we discuss our inverse problem associated with a weaker
class of solutions to the problem \eqref{eq-gov} in a suitable Sobolev space
in time.
First we shall introduce function spaces and notations which are needed
subsequently.
We define the fractional Sobolev space $H^{\alpha}(0,T)$
with the norm in $H^{\alpha}(0,T)$:
$$
\|u\|_{H^{\alpha}(0,T)} := \left(  \|u\|^2_{L^{2}(0,T)}
+ \int_0^T \int_0^T \frac{|u(t)-u(s)|^2}{|t-s|^{1+2\alpha}}dtds
      \right)^{\frac{1}{2}}
$$
(see e.g., \cite{Adams1975sobolev}).
Furthermore, we define the Banach spaces
$$
H_{\alpha}(0,T):=
\left\{ \begin{aligned}
& \{ u \in H^{\alpha}(0,T); u(0)=0\}, \quad \alpha\in(1/2, 1), \\
& \left\{v \in H^{\frac{1}{2}}(0,T); \int^T_0 \frac{|v(t)|^2}{t}dt < \infty
\right\}, \quad \alpha = \frac{1}{2}, \\
& H^{\alpha}(0,T), \quad \alpha\in(0, 1/2),
\end{aligned}\right.
$$
endowed with the norm
$$
\|v\|_{H_{\alpha}(0,T)}:=
\left\{ \begin{aligned}
& \|v\|_{H^{\alpha}(0,T)}, \quad \quad \quad \quad \quad \quad  \quad \quad
\  \quad \alpha \in (0,1 ), \ \alpha \neq \frac{1}{2}, \\
& \left( \|v\|^2_{H^{\frac{1}{2}}(0,T)} + \int^T_0 \frac{|v(t)|^2}{t}dt
\right)^{\frac{1}{2}}, \ \  \alpha = \frac{1}{2}.
\end{aligned}\right.
$$
We introduce the Riemann-Liouville fractional integral operator $J^\alpha$,
$0<\alpha<1$ (see e.g., \cite{Podlubny1998fractional})
$$
J^{\alpha}g(t) = \frac{1}{\Gamma(\alpha)} \int_0^t (t-s)^{\alpha-1}g(s)ds,
\quad 0<t<T,
$$
and we see that $J^\alpha$ is invertible, and we define the time fractional
derivative $\partial_{t}^{\alpha}$ in $H_{\alpha}(0,T)$ by
$$
\partial_{t}^{\alpha}g =(J^{\alpha})^{-1}g, \quad g\in H_{\alpha}(0,T)
=J^{\alpha}L^2(0,T)
$$
(e.g., Gorenflo, Luchko and Yamamoto \cite{GLY15},
Kubica, Ryszewska and Yamamoto \cite{Kubica2020introduction}).
Then we can verify that $\partial_{t}^{\alpha}$ is an extension of the
Caputo derivative $d_{t}^{\alpha}$ to $H_{\alpha}(0,T)$, that is,
$\partial_t^\alpha u = d_t^\alpha u$ in the case when $u\in AC[0,T]$
and $u(0)=0$, and see the details in
\cite{Kubica2020introduction}.

In place of \eqref{eq-gov}, we introduce
the following initial-boundary value problem for the
time-fractional diffusion equation:
\begin{equation}\label{eq-gov'}
\left\{
\begin{alignedat}{2}
& \partial_{t}^{\alpha}u(x,t) = \partial_{x}^{2}u(x,t) + q(x)u(x,t)
&\quad& \mbox{in } L^2(0,T;L^2(0,1)), \\
& \partial_{x}u(0,t) - hu(0,t)=0,   &\quad &\\
&\partial_{x}u(1,t) + H u(1,t)= \eta(t) &\quad& \mbox{in }L^2(0,T),\\
& u \in H_{\alpha}(0,T; L^2(0,1)).
\end{alignedat}
\right.
\end{equation}

This generalizes the problem \eqref{eq-gov} and formulates
the problem for a time-fractional diffusion equation
in the space $L^2(0,T;(0,1))$,
which admits a unique solution in the class $H_{\alpha}(0,T;L^2(0,1))$
for given boundary data $\eta\in H_{\alpha}(0,T)$
(\cite{Rundell2021uniqueness}).
Other choice of function spaces for boundary data and solutions
is possible, but our formulation is consistent within the Sobolev
spaces $H_{\alpha}(0,T)$ as the time regularity both for data and solutions.

Here we interpret $u \in H_{\alpha}(0,T; L^2(0,1))$ in \eqref{eq-gov'} as
the zero initial condition because we can verify that
$u \in H_{\alpha}(0,T;L^2(0,1))$ implies
$u(\cdot,0) =  0$ in the trace sense of $H^{\alpha}(0,T)$ if
$\frac{1}{2} < \alpha < 1$.
For the details, we refer also to \cite{Kubica2020introduction}.

Moreover we introduce an admissible set of potentials and Robin coefficients:
\begin{equation}\label{def-Ad}
\mathcal A := \{ (q,h,H);\ q\in C[0,1], \ -q, h, H \geq 0\}.
\end{equation}
The main purpose of this paper is to study the uniqueness in the
identification of coefficients $q$ and $h$ from an interior one point
measurement for \eqref{eq-gov'}.

Now we are ready to state our first main results.
\begin{thm}\label{thm-unique}
Let $d\in (0,1)$ and $x_0 \in (0,1)$ be given.
Let $u_j(x, t) $ be the solutions of \eqref{eq-gov'} with $(q_j,h_j,H)
\in\mathcal A$, $j=1, 2$.
We assume that $\eta\in H_{\alpha}(0,T)$
does not vanish identically and $q_{1}(x)=q_{2}(x)$, $x\in (d,1)$.
Suppose that
\begin{align*}
& (d,x_0) \in
\{ (d,x_0) \in (0,1)^2; \, 0< d \le x_0 \le 1\} \\
\cup &
\left\{ (d,x_0) \in (0,1)^2; \, 0\le x_0 < \min\{ d, -2d+1\}, \,\,
0<d<\frac{1}{2} \right\}.
\end{align*}
Then $u_{1}(x_{0},t) = u_{2}(x_{0},t)$, $0<t<T$, implies
$$
q_{1}(x)=q_{2}(x),  \ \ x\in [0,1], \quad and \quad h_1=h_2.
$$
\end{thm}

\begin{rmk}
Since $(d,x_0) = (1,1)$ is covered by the range in Theorem 1, we
have the uniqueness.
On the other hand, with $x_0=1$,
the problem is reduced to the inverse problems
on the determination of $h$ and $q(x)$, $0<x<1$ by boundary data, and
the uniqueness can be obtained similarly to \cite{Rundell2021uniqueness}.
In other words, Theorem 1 generalizes the existing result by using a general
choice of observation point $x_0$ for the fractional diffusion equations.
\end{rmk}

Next we consider the case which is not included in Theorem 1.
For the statement, we introduce a operator and notations.
For $q_1 \in C[0,1]$, we define an operator $L(q_1)$ in $L^2(0, 1)$ by
\begin{equation}\label{eq-sdop}
\left\{
\begin{aligned}
&L(q_1) u(x) = - u''(x) - q_1(x)u(x), \quad 0<x<1,\\
& \mathcal{D} (L(q_1)) = \left\{
u \in H^2(0,1);\,  \frac{du}{dx}(0) - hu(0) = \frac{du}{dx}(1)
+ Hu(1) = 0 \right\},
\end{aligned}\right.
\end{equation}
where $(q_1, h, H)$ belongs to the admissible set $\mathcal A$
defined by \eqref{def-Ad}. It is well known that $L(q_1)$ is
self-adjoint in $L^2(0,1)$ and possesses a discrete spectrum, denoted
by $\sigma(L(q_1)):=\{\lambda_{1,n} \}_{n\in \mathbb{N} \cup \{0\}}$
which consists of simple real eigenvalues
and admits the asymptotic behavior $\lambda_{1,n} =O(n^2\pi^2)$
as $n\to\infty$, and
see e.g., \cite{Levitan1991sturm}.
Moreover by $\varphi_{1,n}$, we denote a normalized eigenfunction for
$\lambda_{1,n}$: $L(q_1)\va_{1,n} = \lambda_{1,n}\va_{1,n}$ in $(0,1)$ and
$\Vert \va_{1,n}\Vert:= \int^1_0 \vert \va_{1,n}(x)\vert^2 dx = 1$.

Given a sequence of positive real numbers $I:=\{a_n\}_{n=0}^{\infty}$, let
\begin{equation}\label{def-N}
N_{I}(s) :=  \# \{ n\in \mathbb{N} \cup \{0\}: a_{ n} \leq s \},
\end{equation}
where $\# S$ means the number of
the elements of a set $S$. Furthermore, we define the set
\begin{equation}\label{def-Lambda}
\Lambda:=\{\lambda_{1,n};\varphi_{1,n}(x_0)\neq0\}.
\end{equation}

Now we are ready to state the second main result.
\begin{thm}\label{thm-unique'}
We assume that there exist constants $A \ge 2d$ and $B \geq -\frac{1}{4}-d$
such that
\begin{equation}\label{ineq-Lambda}
N_{\Lambda}(s) \geq A N_{\sigma(L(q_{1}))}(s) + B
\end{equation}
for sufficiently large $s>0$, where
the set $\Lambda$ is defined as \eqref{def-Lambda}
For such $d\in (0,1)$, let
$$
(d,x_0) \in \left\{ (d,x_0)\in (0,1)^2;\,
-2d + 1 < x_0 < d, \,\, \frac{1}{3} < d < \frac{1}{2}\right\}.
$$
Let $u_j(x, t)$ be the solutions of \eqref{eq-gov'} with
$q_j,h_j\in\mathcal A$, $j=1, 2$.
We assume that $q_{1}(\cdot)=q_{2}(\cdot)$ in $(d,1)$,
and $\eta \in H_{\alpha}(0,T)$ does not vanish identically in $(0,T)$.
Then $u_{1}(x_{0},\cdot) = u_{2}(x_{0},\cdot)$ in $(0,T)$ implies
$$
q_{1}(\cdot) = q_{2}(\cdot) \mbox{ in }[0,1] \mbox{ and }h_1 = h_2.
$$
\end{thm}

\begin{rmk}
Theorem 1 asserts the uniqueness for $(d,x_0) \in
\{ (d,x_0) \in (0,1)^2; \, 0< d \le x_0 \le 1\} \cup
\left\{ (d,x_0) \in (0,1)^2; \, 0\le x_0 < \min\{ d, -2d+1\}, \,\,
0<d<\frac{1}{2} \right\}$.
On the other hand, if we can choose $A > \frac{2}{3}$, then we can
directly verify that the uniqueness range of $(d,x_0)$ in Theorem 2 is
disjoint with the uniqueness range in Theorem 1,
and so we can extend the range of $(d,x_0)$ for the uniqueness.
For $(d, x_0)$ which Theorems 1 and 2 do not cover,
we do not know the uniqueness.
\end{rmk}

We provide a sufficient condition for \eqref{ineq-Lambda}.
\\
{\bf Lemma 1.}
{\it
We assume that there exists a constant $A > 0$ such that
\begin{equation}\label{2.7}
\lim\inf_{s\to\infty} N_{\Lambda}(s)s^{-\frac{1}{2}} > \frac{A}{\pi}.
\end{equation}
Then \eqref{ineq-Lambda} follows with
$d \le \frac{A}{2}$, so that the uniqueness for the inverse problem
holds.
}
\\
{\bf Proof.}
We have $\sqrt{\lambda_{1,n}} = n\pi + O(\frac{1}{n})$ as
$n \to \infty$ (e.g., p.8 in \cite{Levitan1991sturm}).
We remark that in \cite{Levitan1991sturm}, as the $x$-interval,
$(0,\pi)$ is considered in place of
$(0,1)$, but the change in $x$ can yield the asymptotics for the
$x$-interval $(0,1)$.
Then $\lambda_{1,n} \le s$ is equivalent to
$n + O(\frac{1}{n}) \le \frac{1}{\pi}s^{\frac{1}{2}}$ as $n\to \infty$.
Hence
$$
\frac{1}{\pi}s^{\frac{1}{2}} + O(\frac{1}{n}) - 1 \le n
\le \frac{1}{\pi}s^{\frac{1}{2}} + O(\frac{1}{n}),
$$
which implies
$$
\frac{1}{\pi}\frac{s^{\frac{1}{2}}}{n} + O(\frac{1}{n^2}) - \frac{1}{n}
\le 1 \le \frac{1}{\pi}\frac{s^{\frac{1}{2}}}{n} + O(\frac{1}{n^2}).
$$
Therefore, $\lim_{n\to \infty} \frac{1}{\pi}\frac{s^{\frac{1}{2}}}{n}
= 1$, and so $O(\frac{1}{n}) = O(\frac{1}{s^{\frac{1}{2}}})$ as
$n\to \infty$, that is, $s \to \infty$.
Hence,
$$
\frac{1}{\pi}s^{\frac{1}{2}} + O(\frac{1}{s^{\frac{1}{2}}}) - 1
\le N_{\sigma(L(q_1))}(s)
\le \frac{1}{\pi}s^{\frac{1}{2}} + O(\frac{1}{s^{\frac{1}{2}}})
$$
as $s \to \infty$.
Consequently,
\begin{equation}\label{2.8}
N_{\sigma(L(q_1))}(s) = \frac{1}{\pi}s^{\frac{1}{2}}
+ o(s^{\frac{1}{2}}) \quad \mbox{as $s \to \infty$.}
\end{equation}
The assumption yields
$$
\lim\inf_{s\to\infty} N_{\Lambda}(s)s^{-\frac{1}{2}}
> \frac{A}{\pi} + 2\ep \quad \mbox{for sufficiently small
$\ep > 0$}.
$$
Then we can find large $s_1>0$ such that
$$
N_{\Lambda}(s)s^{-\frac{1}{2}} > \frac{A}{\pi} + 2\ep \quad
\mbox{for all $s \ge s_1$}.
$$
Let $B \in \mathbb{R}$ be given arbitrarily.
For these constants $\ep>0$ and $B$, we can choose large $s_2>0$ such that $s_2 \ge s_1$ and
$$
A\frac{o(s^{\frac{1}{2}})}{s^{\frac{1}{2}}} < \ep, \quad
Bs^{-\frac{1}{2}} < \ep
$$
for all $s \ge s_2$.
Hence,
$$
N_{\Lambda}(s)s^{-\frac{1}{2}} > \frac{A}{\pi}
+ A\frac{o(s^{\frac{1}{2}})}{s^{\frac{1}{2}}} + Bs^{-\frac{1}{2}}
= A\left( \frac{1}{\pi} + \frac{o(s^{\frac{1}{2}})}{s^{\frac{1}{2}}}\right)
+ Bs^{-\frac{1}{2}} \quad \mbox{for $s \ge s_2$},
$$
that is,
$$
N_{\Lambda}(s) \ge A\left( \frac{1}{\pi}s^{\frac{1}{2}}
+ o(s^{\frac{1}{2}})\right) + B
= AN_{\sigma(L(q_1))} + B \quad \mbox{for $s \ge s_2$},
$$
which is \eqref{ineq-Lambda}.
$\blacksquare$
\\

Now we discuss two example in terms of Lemma 1.
\\
{\bf Example.}
For simplicity, we consider $q_1 \equiv 0$ and $h=H=0$.  Then
$\sigma(L(q_1)) = \{ n^2\pi^2\}_{n\in \NZ}$ and
$\va_{1,n}(x) =
\left\{ \begin{array}{rl}
& 1, \quad n=0, \\
& \sqrt{2}\cos n\pi x, \quad n\in \N.
\end{array}\right.$
\\
{\bf Case (i).}
Let $x_0 \not\in \mathbb{Q}$.  Then $\va_{1,n}(x_0) \ne 0$ for all
$n \in \NZ$, and so $\Lambda = \{\lambda_{1,n}\}_{n\in \NZ}$.
Hence, similarly to \eqref{2.8}, we see that
$$
N_{\Lambda}(s) \ge \frac{s^{\frac{1}{2}}}{\pi} + o(s^{\frac{1}{2}}).
$$
Therefore, $\liminf_{s\to\infty} N_{\Lambda}(s)s^{-\frac{1}{2}}
\ge \frac{1}{\pi}$ and \eqref{2.7} holds for any $A < 1$.
Consequently, in Theorem 2, we can choose $d < \frac{1}{2}$.
Thus, setting $d = \frac{1}{2} - \ep$ with sufficiently small
$\ep > 0$, if $q_2(x) = q_1(x) = 0$ for $\frac{1}{2} - \ep
< x < 1$, then
$$
u_1(x_0,t) = u_2(x_0,t) \quad \mbox{for $0<t<T$ with
some $x_0 \in \left(2\ep,\, \frac{1}{2} - \ep\right)$}
$$
implies $q_2(x) = 0$ for $0<x < 1$.
\\
{\bf Case (ii).}
Let $x_0 = \frac{1}{2}$.  Then
$\Lambda = \{ 4n^2\pi^2;\, n \in \NZ\}$.
Similarly to \eqref{2.8}, we can obtain
$$
N_{\Lambda}(s) \ge \frac{s^{\frac{1}{2}}}{2\pi} + o(1).
$$
Therefore, \eqref{2.7} is satisfied if $A < \frac{1}{2}$.
Hence, in order that $A \ge 2d$ is satisfied, we have to
assume that $d < \frac{1}{4}$ and so for the uniqueness for the
inverse problem, a stronger condition than Case (i) is necessary, that is,
we must assume that $q_2(x) = 0$ for $\frac{1}{4}-\ep < x < 1$
with some constant $\ep > 0$.
$\blacksquare$

\begin{rmk}
The corresponding uniqueness results in Theorems \ref{thm-unique} and
\ref{thm-unique'} remains valid if one take other boundary conditions e.g.,
Dirichlet boundary condition or Neumann boundary condition in \eqref{eq-gov'}.
The proof is similar, and so we omit it.
\end{rmk}

The remainder of the paper is organized into three sections. Section 3 compiles
definitions and preliminary outcomes on the spatial differential operator,
along with several lemmas essential for addressing the forward problem.
Section 4 is dedicated to the proofs of Theorems \ref{thm-unique} and
\ref{thm-unique'}.  Finally, Section 5 provides concluding remarks.
Throughout this paper, we use the symbol $C$ to represent a generic constant,
which may vary line by line.

\section{Preliminaries}

We recall that the operator $L(q)$ in $L^2(0, 1)$ is defined by
\eqref{eq-sdop}.
We consider the initial-value problems
\begin{equation}\label{eigen-h}
\left \{\begin{aligned}
& L(q)\varphi_q(x;\lambda) = \lambda(q) \varphi_q(x;\lambda) \quad
\mbox{in }(0,1), \\
& \varphi_q(0;\lambda)=1, \ \ \ \varphi'_q(0;\lambda) = h, \\
\end{aligned}\right.
\end{equation}
and
\begin{equation}\label{eigen-H}
\left \{\begin{aligned}
& L(q)\psi_q(x;\lambda) = \lambda(q) \psi_q(x;\lambda)\quad \mbox{in }(0,1),\\
& \psi_q(1;\lambda) = 1, \ \ \ \psi'_q(1;\lambda) = -H.\\
\end{aligned}\right.
\end{equation}
We sometimes omit $q$ in the above notations, and
for examples we write $\lambda=\lambda(q)$ and $\varphi(x;\lambda)
=\varphi_q(x;\lambda)$ for short. Moreover, in the case of
$\lambda = \lambda_n = \lambda_n(q)$, we rewrite $\varphi_n(x)
= \varphi_q(x;\lambda_n)$ and $\psi_n(x) = \psi_q(x;\lambda_n)$
if no conflicts may occur.

It is well known that $\varphi(x;\lambda)$ and
$\psi(x;\lambda)$ are entire functions of $\lambda$ for each $x\in [0,1]$,
whereas $\varphi_n(x)$ and $\psi_n(x)$ are the eigenfunctions of the operator
$L(q)$ corresponding to the eigenvalue $\lambda_n$. Moreover, $\psi_n$ and
$\varphi_n$ are linearly dependent, that is, the relation
\begin{equation}\label{linear_dependent}
\psi_n(x) = k_n \varphi_{n}(x),\quad \forall \ x\in (0,1)
\end{equation}
holds true. Here $k_n = \frac{1}{\varphi_n(1)}$ is neither zero nor $\infty$,
which is called the norming constant corresponding to $\lambda_{n}$.
In addition, $\varphi(x;\lambda)$, $\varphi'(x;\lambda)$ satisfy the following
asymptotic behaviour as $|\lambda|\rightarrow \infty$ uniformly
with respect to $x\in[0,1]$:
\begin{equation}\label{asymp-phi}
\begin{aligned}
\varphi(x;\lambda) &= \cos(\sqrt{\lambda}x)
+ O \left(\frac{e^{|\mathrm{Im} \sqrt{\lambda}|x}}{|\sqrt{\lambda}|}\right),
                                                \\
\varphi'(x;\lambda) &= -\sqrt{\lambda}\sin(\sqrt{\lambda}x)
+ O (e^{|\mathrm{Im} \sqrt{\lambda}| x}),
\end{aligned}
\end{equation}
The above assertions can be found in e.g., \cite{Levitan1991sturm}.

Next, we divide the original problem \eqref{eq-sdop} into two problems
with respect to the intervals $(0,x_0)$ and  $(x_0,1)$:
\begin{equation}\label{def-mu-}
L(q_{1,-}) y :=
\left \{\begin{aligned}
& -y^{''} + q_{1}(x)y = \mu_{n}^- y,   \ \  x\in (x_{0}, 1), \\
& y'(0) - hy(0) = 0,  \  y(x_0)=0 \\
\end{aligned}\right.
\end{equation}
and
\begin{equation}\label{def-mu+}
L(q_{1,+}) y := \left \{\begin{aligned}
& -y^{''} + q_{1}(x)y = \mu_{n}^{+} y,   \ \  x\in (x_{0}, 1), \\
& y(x_{0})=0,  \  y^{'}(1) + H y(1)=0. \\
\end{aligned}\right.
\end{equation}
We also know from \cite{freiling2001inverse} that $\mu_{\pm}$ admits
the following asymptotic behavior:
\begin{lem}\label{lem-mu}
Let $\mu_\pm$ be defined as in \eqref{def-mu-} and \eqref{def-mu+}.
Then the asymptotic expansions hold true:
$$
\begin{aligned}
\sqrt{\mu_{n}^{-}} &= \frac{(n+\frac{1}{2})\pi}{x_{0}} + O\left(\frac{1}{n}
\right),\\
\sqrt{\mu_{n}^{+}} &= \frac{(n+\frac{1}{2})\pi}{1-x_{0}} + O\left(\frac{1}{n}
\right),
\end{aligned}
\quad n\to\infty.
$$
\end{lem}

On the basis of the above lemma, recalling the definitions of the sets
$N_\Lambda(s)$ and $\Lambda$, it follows that
\begin{equation}\label{def-Lambda'}
\Lambda^c :=  \sigma(L(q_{1})) \setminus \Lambda = \{\lambda_{1,n}:
\phi_{1,n}(x_{0}) = 0 \}.
\end{equation}
Moreover, it is not difficult to see that
$$
\Lambda^{c} \subset \{ \mu_{n}^{-}\}_{n=1}^{\infty}
\cap \{ \mu_{n}^{+}\}_{n=1}^{\infty},
$$
where $\mu_{n}^{-}$, $\mu_{n}^{+}$ are defined by \eqref{def-mu-} and
\eqref{def-mu+} respectively. We have
\begin{lem}
Let $\Lambda$ and $N_\Lambda(s)$ be defined in \eqref{def-Lambda}
and \eqref{def-N} respectively. Then $N_\Lambda(s) \ge \max\{1-x_0,x_0\}
\frac{\sqrt{s}}{\pi}$ is valid for sufficiently large $s>0$.
\end{lem}
\begin{proof}
From the definition of the set $N_{\Lambda^c}(s)=  \# \{ \lambda_{1,n} \in
\Lambda^c : \lambda_{1, n} \leq s \}$, it is not difficult to
check the identity
$$
N_{\Lambda^c}(s) = \displaystyle N_{\sigma(L(q_{1,-})) \cap
\sigma(L(q_{1,+}))} (s)
$$
for any $s>0$, which implies
$$
N_{\Lambda^c}(s) \leq \displaystyle \min\{1-x_{0}, x_{0}\}\frac{\sqrt{s}}{\pi}
\mbox{ for sufficiently large }s>0.
$$
Here the inequality is due to the following estimates:
$$
\begin{aligned}
\displaystyle N_{\sigma(L(q_{1,-})) \cap \sigma(L(q_{1,+}))} (t)
 \le& \min\{1-x_{0}, x_{0}\} N_{\sigma(L(q_{1}))} (t)   \\
 \le& \min\{1-x_{0}, x_{0}\} \frac{\sqrt{t}}{\pi}\quad \mbox{for sufficiently
large }s>0.
\end{aligned}
$$
Finally, we can obtain
\begin{equation}\label{esti-N}
N_\Lambda(s) \geq \displaystyle \left(1 - \min\{1-x_{0}, x_{0}\}
\right)\frac{\sqrt{s}}{\pi}
\end{equation}
for sufficiently large $s>0$. This completes the proof of the lemma.
\end{proof}

\subsection{Auxiliary functions}
In this part, we introduce several auxiliary functions related to the problems
\eqref{eigen-h}, \eqref{eigen-H}, \eqref{def-mu-} and \eqref{def-mu+}.
We define a function $\Delta(\lambda)$ by
\begin{equation}\label{def-Delta}
\Delta(\lambda) := -\varphi{'}(1; \lambda)- H\varphi(1; \lambda).
\end{equation}
Then from \cite{freiling2001inverse}, $\lambda_n$ is a zero of
$\Delta(\lambda)$: $\Delta(\lambda_n) = 0$.
Moreover we can calculate $\dot\Delta(\lambda_n)$ as follows
\begin{equation}\label{eq-Delta'}
\dot{\Delta}(\lambda_n) := \left(\frac{d\Delta}{d\lambda}\right)(\lambda_n)
= -k_{n}\beta_n,
\end{equation}
where
\begin{equation}\label{def-beta}
\beta_{n} := \int_0^1 |\varphi_n(x)|^2 dx.
\end{equation}

The Weyl $m_{-}-$function is defined by
\begin{equation}\label{def-m_}
m_{-}(x,\lambda) = -\frac{\varphi'(x;\lambda)}{\varphi(x;\lambda)},
\end{equation}
for any $x\in[0,1]$ (\cite{F2000inverse}).
From \cite{danielyan1991asymptotic}, we see
\begin{lem}\label{lem-m_}
Assume $q\in C[0,1]$.  Then the Weyl $m_-$-function defined in \eqref{def-m_}
admits the asymptotic expansion:
\begin{equation}\label{asymp-m}
m_-(x,\lambda) =  -i(\sqrt{\lambda})^{-1}\left(1 + o(\lambda^{-1/2})\right),
\end{equation}
uniformly in $x\in[\delta, 1-\delta]$ for $\delta>0$ as $|\lambda|\rightarrow
\infty$ in any sector $\Lambda(\varepsilon):=\{\lambda \in \mathbb C:
\varepsilon< Arg(\lambda)<\pi-\varepsilon, \varepsilon>0\}$.
\end{lem}

We next define the function $U(x,\lambda)$ by
\begin{equation}\label{def-U}
U(x;\lambda) := \left|
\begin{array}{cc}
\varphi_1(x;\lambda)  &   \varphi_2(x;\lambda)     \\
\varphi_1'(x;\lambda)  &    \varphi_2'(x;\lambda)
\end{array}
\right|.
\end{equation}
Here and henceforth
$\vert \cdot \vert$ denotes determinants of $2\times 2$-matrices under
consideration and we write
$\varphi_j(x;\lambda):=\varphi_{q_j}(x;\lambda)$, $j=1,2$.

Then we have
\begin{lem}\label{lem-U-multi2}
Let $\varphi_{1,n}$ and $\varphi_{2,m}$ be the eigenfunctions of the problem
\eqref{eigen-h} corresponding to the eigenvalues $\lambda_{1,n}$ and
$\lambda_{2,m}$ with respect to $q=q_1$ and $q=q_2$ satisfying
\begin{equation}\label{condi-varphi}
\frac{\varphi_{1,n}(1)\varphi_{1,n}(x_0)}{\|\varphi_{1,n}\|^2}
= \frac{\varphi_{2,m}(1)\varphi_{2,m}(x_0)}{\|\varphi_{2,m}\|^2}.
\end{equation}
Assume $q_1=q_2$ in $[d,1]$ with $d\le x_0$.
Then $\lambda_{1,n}$ is a zero of $ U(d;\cdot)$ with multiplicity
at least $2$ providing that $\lambda_{1,n} = \lambda_{2,m}$.
\end{lem}
\begin{proof}
It is sufficient to show that $U(d,\lambda_{1,n})=\dot U(d,\lambda_{1,n})=0$.
For this, using the assumption on $q_1(x) = q_2(x)$ for $x \in [d,1]$, we can
show that
\begin{equation}\label{eq-U}
\begin{aligned}
U(d;\lambda) = & U(1;\lambda) - \int_d^1 \frac{d}{dx}U(x;\lambda) dx \\
 = & U(1;\lambda) + \int_d^1 \left(q_1 - q_2\right)(x)(\varphi_1
\cdot \varphi_2)(x;\lambda)dt = U(1;\lambda) .
\end{aligned}
\end{equation}
From \eqref{eq-U} and the definition of $\varphi_{j}$ in \eqref{eigen-h} and
\eqref{eigen-H}, $j=1,2$, it follows that
\begin{equation}\label{multi-1}
\begin{aligned}
U(d;\lambda_{1,n})
& =
\left| \begin{array}{cc}
    \varphi_1(1;\lambda_{1,n}) &   \varphi_2(1;\lambda_{1,n})     \\
    \varphi_1'(1;\lambda_{1,n}) &    \varphi_2'(1;\lambda_{1,n})
\end{array} \right| \\
& =
\left| \begin{array}{cc}
    \varphi_{1}(1;\lambda_{1,n}) &   \varphi_{2}(1;\lambda_{1,n})     \\
    -H\varphi_{1}(1;\lambda_{1,n}) &    -H\varphi_{2}(1;\lambda_{1,n})
\end{array} \right|  = 0
\end{aligned}
\end{equation}
is valid for any eigenvalue $\lambda_{1,n}$ of the problem \eqref{eq-sdop}.
Therefore, we finish the proof of $U(d;\lambda_{1,n}) = 0$.

Furthermore, noting the linear dependency \eqref{linear_dependent}
and the definition $\beta_n$ in \eqref{def-beta}, we obtain
$$
\frac{\varphi_j(x;\lambda_{j,n})}{\beta_{j, n}^{1/2}} =
\frac{\psi_j(x;\lambda_{j, n})}{k_{j,n} \beta_{j, n}^{1/2}},\quad j = 1, 2.
$$
Therefore, by the assumption \eqref{condi-varphi}, we see
$$
 \frac{\psi_{1}(x_0;\lambda_{1,n})}{k_{1,n}^{2} \cdot \beta_{1,n}}
 = \frac{\psi_{2}(x_0;\lambda_{2, m})}{k_{2, m}^{2} \cdot \beta_{2, m}}.
$$
In addition, from $\lambda_{1,n} = \lambda_{2,m}$ and $q_1(\cdot)
= q_2(\cdot)$ in $[d, 1]$, and $x_{0} \geq d$, we conclude from the uniqueness
of the Cauchy problem \eqref{eigen-H} that
$$
\psi_{1}(x;\lambda_{1,n}) =  \psi_{2}(x;\lambda_{1,n}),\quad x\in[x_0,1].
$$
Hence
$$
k_{1,n}^{2} \cdot \beta_{1,n} = k_{2, m}^{2} \cdot \beta_{2, m},
$$
with which \eqref{eq-Delta'} finally implies that
\begin{equation}\label{eq-kn}
k_{1,n} \left( \dot{\varphi}_1'(1;\lambda_{1,n})
+ H \dot{\varphi}_1(1;\lambda_{1,n}) \right)
= k_{2,m}\left(\dot{\varphi}_2'(1;\lambda_{1,n})
+ H \dot{\varphi}_2(1;\lambda_{1,n})\right).
\end{equation}

Now noting the definition of $\dot U$, it follows from the linear dependency
\eqref{linear_dependent} that
\begin{align*}
\dot{U}(d,\lambda_{1,n}) =&
\left|
\begin{array}{cc}
    \dot{\varphi}_1(1;\lambda_{1,n}) &   \dot{\varphi}_2(1;\lambda_{1,n})    \\
    \varphi_1'(1;\lambda_{1,n}) &    \varphi_2'(1;\lambda_{1,n})
\end{array}
\right|
+ \left|
\begin{array}{cc}
    \varphi_1(1;\lambda_{1,n}) &   \varphi_2(1;\lambda_{1,n})     \\
    \dot{\varphi}_1'(1;\lambda_{1,n}) &    \dot{\varphi}_2'(1;\lambda_{1,n})
\end{array}
\right|
\\
= &\left|
\begin{array}{cc}
    \dot{\varphi}_1(1;\lambda_{1,n}) &   \dot{\varphi}_2(1;\lambda_{1,n})    \\
    \frac{\psi_1'(1;\lambda_{1,n})}{k_{1,n}} &
\frac{\psi_2'(1;\lambda_{1,n})}{k_{2,n}}
\end{array}
\right|
 + \left|
\begin{array}{cc}
    \frac{\psi_1(1;\lambda_{1,n})}{k_{1,n}} &
\frac{\psi_2(1;\lambda_{1,n})}{k_{2,n}} \\
    \dot{\varphi}_1'(1;\lambda_{1,n}) &    \dot{\varphi}_2'(1;\lambda_{1,n})
\end{array}
\right|,
\end{align*}
from which we further see by using the boundary conditions of
$\psi_j$, $j=1,2$ that
\begin{equation}\label{multi-2}
\begin{aligned}
\dot U(d,\lambda_{1,n}) &=
\left|
\begin{array}{cc}
    \dot{\varphi}_1(1;\lambda_{1,n}) &   \dot{\varphi}_2(1;\lambda_{1,n})    \\
    \frac{-H}{k_{1,n}} &    \frac{-H}{k_{2,n}}
\end{array}
\right|
 + \left|
\begin{array}{cc}
    \frac{1}{k_{1,n}} &   \frac{1}{k_{2,n}} \\
    \dot{\varphi}_1'(1;\lambda_{1,n}) &    \dot{\varphi}_2'(1;\lambda_{1,n})
\end{array}
\right|
\\
&= \left|
\begin{array}{cc}
    \frac{1}{k_{1,n}} &   \frac{1}{k_{2,n}} \\
    H\varphi_1'(1;\lambda_{1,n}) + \dot{\varphi}_1'(1;\lambda_{1,n})
&    H\varphi_2'(1;\lambda_{1,n}) + \dot{\varphi}_2'(1;\lambda_{1,n})
\end{array}
\right| = 0,
\end{aligned}
\end{equation}
where the last line is due to \eqref{eq-kn}. These two identities
\eqref{multi-1} and \eqref{multi-2} illustrate that $\lambda_{1,n}$
is a zero of $U(d,\cdot)$ with multiplicity at least $2$.
We complete the proof of the lemma.
\end{proof}

Recalling the eigenvalue $\lambda_{1,n}\in \sigma(L(q_1))$ defined in
\eqref{eq-sdop}, we define
$$
g_{\sigma}(\lambda) = \prod_{n=0}^{\infty} \left(1 -
\frac{\lambda}{\lambda_{1,n}}\right),\quad \lambda\in\mathbb C,
$$
and on the basis of the set $\Lambda$ in \eqref{def-Lambda}, we further
introduce the function $g_\Lambda$ as follows:
$$
g_\Lambda(\lambda) := \prod_{\lambda_{1,n} \in \Lambda}
\left(1-\frac{\lambda}{\lambda_{1,n}}\right),\quad \lambda\in\mathbb C.
$$

Based on the above two functions, we further consider the function
\begin{equation}\label{def-F}
F(\lambda) = \frac{U(d;\lambda)}{g_\Lambda^2(\lambda)}.
\end{equation}
From the asymptotic expansion of the function $m_-$ in Lemma \ref{lem-m_},
we obtain the following lemma.
\begin{lem}\label{lem-F}
Under the same assumptions in Lemma \ref{lem-U-multi2}, the function
$F(\cdot)$ defined by \eqref{def-F} is an entire function satisfying
$$
\left|F(\lambda)\right|
\leq  C e^{C|\lambda|^{\frac12}}\quad \mbox{for all $|\lambda|
= \pi^2\left(n+\frac12\right)^2$} .
$$
Moreover, if $\lambda=iy$, $y\in\mathbb R$, then
$F(iy)$ admits the asymptotic behavior
$$
F(iy)\to 0,\mbox{ as }y\to\infty.
$$
\end{lem}
\begin{proof}
We first assert that $F(\cdot)$ is an entire function. Indeed, from the
expression of the function $g_\Lambda$, we see that the zeros of $g_\Lambda^2$
are also zeros of the function $U$. Now since $\lambda_{1,n} \in \Lambda$ is
a zero of $U(d;\lambda)$ with multiplicity at least $2$,
noting that $g_\Lambda$ has only a simple zero point, we can see
that all the isolated singularity points are removable. Therefore $F(s)$ is
analytic on the whole complex plane.

Now from the asymptotic behavior of $\varphi(d;\lambda)$ and
$\varphi'(d;\lambda)$ (e.g., \cite{Levitan1991sturm}),
and Theorem B.2 in \cite{F2000inverse}, we conclude that there exist constants $C_1, C_2>0$ such that
\begin{equation}\label{esti-F}
\begin{aligned}
\left|F(\lambda)\right|
=& \left|\frac{U(d;\lambda)}{g_{\sigma}^2(\lambda)}\cdot\frac{g_{\sigma}^2
(\lambda)}{g_\Lambda^2(\lambda)}\right|  \\
=& \left| \frac{\varphi_1(d;\lambda) \varphi_2'(d;\lambda)
- \varphi_2(d;\lambda) \varphi_1'(d;\lambda)}{g_{\sigma}^2(\lambda)}\right|
\cdot \frac{g_{\sigma}^2(\lambda)}{g_\Lambda^2(\lambda)}
\leq  C_1 e^{C_2|\lambda|^{1/2}}
\end{aligned}
\end{equation}
for any $\lambda$ with $|\lambda|=\pi^2 (n+\frac12)^2$.
In addition, from the definitions of $U(x;\lambda)$ and the Weyl function
$m_-(x;\lambda)$ in \eqref{def-U} and (3.12),
it follows that
\begin{align*}
U(d;\lambda)
& = \varphi_1(d;\lambda) \varphi_2'(d;\lambda) - \varphi_2(d;\lambda)
\varphi_1'(d;\lambda) \\
& = \frac{-\varphi_1'(d;\lambda)}{m_{1,-}(d;\lambda)}\varphi_2'(d;\lambda)
+ \frac{\varphi_2'(d;\lambda)}{m_{2,-}(d;\lambda)}\varphi_1'(d;\lambda) \\
& = \varphi_1'(d;\lambda) \varphi_2'(d;\lambda) \left(\frac{1}{m_{2,-}
(d;\lambda)} - \frac{1}{m_{1,-}(d;\lambda)} \right).
\end{align*}
Consequently, also by the notation of $F(\lambda)$, we obtain
\begin{equation}\label{eq-F}
F(\lambda) = \frac{\varphi_1'(d;\lambda) \varphi_2'(d;\lambda)}
{g_\Lambda^2(\lambda)}\left(\frac1{m_{2,-}(d;\lambda)}
- \frac1{m_{1,-}(d;\lambda)}\right).
\end{equation}
Moreover, applying Hadamard's factorization theorem (e.g., p.12 in  \cite{freiling2001inverse})
to (3.9), we see
$$
g_{\sigma}(\lambda) =
C\left(-\varphi'_1(1;\lambda) - H \varphi_1(1;\lambda)\right).
$$
Moreover, the definitions of $g_{\sigma}$ and $g_\Lambda$  yields
$$
g_\Lambda(\lambda)
= \frac{g_{\sigma}(\lambda)}{ \prod_{\lambda_{1,n} \in \Lambda^c}
\left(1-\frac{\lambda}{\lambda_{1,n}}\right)}.
$$
By using a similar argument as in \cite[pp. 2784]{F2000inverse}, it follows
from \eqref{asymp-phi} and \eqref{esti-N} that
\begin{equation}\label{asymp-g}
\left|g_\Lambda(iy)\right| \geq C \left| g_{\sigma}(iy)\right|^d
\geq C |y|^{\frac d2} \exp\left\{d {\rm Im}(\sqrt{i})|y|^{\frac12} \right\},
\quad \mbox{as }y\to\infty. 
\end{equation}
Furthermore, from the asymptotic \eqref{asymp-phi} of the eigenfunctions,
we obtain
\begin{equation}\label{asymp-phi'}
\left|\varphi_j'(d;iy)\right| \le C |y|^{\frac12}
\exp\left\{d {\rm Im} (\sqrt{i})|y|^{\frac12} \right\}, \quad j=1,2.
\end{equation}
Moreover, from the asymptotic \eqref{asymp-m} of the $m_{-}-$functions, we have
\begin{equation}\label{asymp-m12}
\left| \frac1{m_{2,-}(d;iy)} - \frac1{m_{1,-}(d;iy)} \right|
= o\left(\frac1{|y|}\right),  \quad \mbox{as }y\to\infty.
\end{equation}
Consequently, by \eqref{eq-F} and collecting all the above estimates
\eqref{asymp-g}--\eqref{asymp-m12}, we infer that
\begin{equation}\label{esti-Fy}
\begin{aligned}
\left|F(iy)\right| =& \displaystyle   \frac{\left|\varphi_1'(d; iy)
\varphi_2'(d;iy)\right|}{\left|g_{\Lambda}^2(iy)\right|}\left|
\frac1{m_{2,-}(d;iy)} - \frac1{m_{1,-}(d;iy)}\right|   \\
\leq & C \frac{|y|}{|y|^d} \times \frac{\exp\left\{2d {\rm Im}
(\sqrt{i})|y|^{\frac12} \right\}} { \exp \left\{2d {\rm Im} (\sqrt{i})
|y|^{\frac12} \right\}} \times o\left(\frac1{|y|} \right) = o(1),
\quad \mbox{as }y\to\infty.
\end{aligned}
\end{equation}
This completes the proof of the lemma.
\end{proof}

\subsection{Solution representation of forward problem}
Let $e_{n}(x; q) := \varphi_{ n}(x; q) / \| \varphi_{n}(\cdot\,;q) \|$ and
let $\{\lambda_{n}(q), e_{n}(x; q) \}_{n\in \mathbb{N}\cup \{0\}}$
be the eigensystem of the operator $L(q)$.
We have the following two lemmata by \cite{Rundell2021uniqueness}, which are
crucial in our subsequent arguments.
\begin{lem}\label{lem-sol}
Let $(q,h,H)\in\mathcal A$ and $\eta\in H_{\alpha}(0,T)$. Then the initial
boundary value problem \eqref{eq-gov'} admits a unique solution
$u\in H_\alpha(0, T; L^{2}(0,1)) \cap L^2(0, T; H^2(0,1))$ which can be
represented as
$$
u(x,t) =\ \sum_{n=0}^{\infty} \left(\int_0^t s^{\alpha-1}
E_{\alpha,\alpha}( -\lambda_n s^\alpha) \eta(t-s) ds \right) e_{n}(x;q)
e_{n}(1;q),\quad (x,t)\in(0,1)\times(0,T).
$$
\end{lem}
\begin{lem}\label{lem-K}
The series
\begin{equation}\label{def-K}
K(x,t):=\sum_{n=0}^{\infty}e_{n}(x;q)e_{n}(1;q)\int_{0}^{t}s^{\alpha-1}
E_{\alpha,\alpha}(-{\lambda_{n}}s^{\alpha})ds
\end{equation}
is uniformly convergent for $(x,t)\in[0,1]\times[0,T]$, and $K(x,\cdot)\in
L^\infty(0,\infty)$ for any $x\in[0,1]$.

Moreover, the function $K$ is analytic in $t>0$ for arbitrarily fixed
$x\in[0,1]$.
For the solution $u(x,t)$ of \eqref{eq-gov'} in Lemma
\ref{lem-sol}, we have
$$
\int _0^t u(x,s)ds = \left(K(x,\cdot) * \eta\right)(t)\quad\mbox{for
all $(x,t)\in[0,1] \times [0,T]$}.
$$
Here $*$ denotes the Laplacian convolution operation.
\end{lem}
The proofs are found in \cite{Rundell2021uniqueness}.
\section{Proof of Theorems 1 and 2}

\begin{proof}[\bf Proof of Theorem \ref{thm-unique}]
Let $\varphi_j(x;\lambda_n)$ and $\psi_j(x;\lambda_n)$ be the solutions
of \eqref{eigen-h} and \eqref{eigen-H} with $q_j$, $j=1,2$ respectively.
By $u_i( x,t)$ we denote the solutions of \eqref{eq-gov'} with $q_j$, $j=1,2$.
In terms of the observation assumption, we see
$$
u_1(x_0, \cdot) = u_2(x_0, \cdot) \mbox{ in }(0,T).
$$
Therefore,
$$
\int _0^t u_1(x_0,s)ds = \int _0^t u_2(x_0,s)ds,  \quad 0<t<T,
$$
Hence, Lemma \ref{lem-K} implies that
$$
\left(K_1(x_0, \cdot) * \eta\right)(t) = \left(K_2(x_0, \cdot) * \eta
\right)(t), \quad 0<t<T.
$$
Thus
$$
(K_1-K_2)(x_0, \cdot) * \eta = 0 \quad \mbox{in }(0,T).
$$
Since $\eta(t) \not \equiv 0$, the Titchmarsh convolution theorem (see e.g.,
\cite{Titchmarsh1926zeros}) implies the existence of $t_0>0$ such that
$K_1(x_0, t) = K_2(x_0, t)$,  $t\in(0, t_0)$. Now we conclude from the
$t$-analytic of $K_i(x_0, t)$, $i=1, 2$, in Lemma \ref{lem-K} that
$$
K_1(x_0, t) = K_2(x_0, t),  \quad  t>0,
$$
that is
\begin{equation}\label{obser-series}
\begin{aligned}
&\sum_{n=0}^\infty \frac{e_n(x_0;q_1)e_n(1;q_1)}{\lambda_n(q_1)}
\Big(1 - E_{\alpha,1}(-{\lambda_n(q_1)} t^\alpha)\Big)\\
=&\sum_{n=0}^\infty \frac{e_n(x_0;q_2)e_n(1;q_2)}{\lambda_n(q_2)}
\Big(1 - E_{\alpha,1}(-{\lambda_n(q_2)} t^\alpha)\Big).  
\end{aligned}
\end{equation}
Henceforth, we also use $\lambda_{j,n}:=\lambda_n(q_j)$ and $\varphi_{j,n}
:=\varphi_n(x;q_j)$, $j=1,2$ for convenience.
By the asymptotic behaviour of the Mittag-Leffler function $E_{\alpha,1}(z)$
(see e.g., \cite{Podlubny1998fractional}), we can obtain
\begin{equation}\label{asymp-ml}
E_{\alpha,1}(-\lambda_{j,n} t^\alpha) = \frac{1}{\Gamma(1-\alpha)}
\frac{1}{\lambda_{j,n} t^\alpha} + O\left(\frac{1}{\lambda_{j,n}^2t^{2\alpha}}
\right), j = 1,2,\quad \mbox{as }t\to\infty.
\end{equation}
Now, substituting the expressions \eqref{asymp-ml} into \eqref{obser-series},
we obtain
\begin{equation}\label{obser-asymt}
\begin{aligned}
&\sum_{n=0}^{\infty} \frac{e_n(x_0;q_1)e_n(1;q_1)}{\lambda_{1,n}}
- \frac{1}{\Gamma(1-\alpha)} \sum_{n=0}^\infty
\frac{e_n(x_0;q_1)e_n(1;q_1)}{\lambda^2_{1,n}} \frac{1}{t^\alpha}
+ O\left(\frac{1}{t^{2\alpha}}\right) \\
 =& \sum_{n=0}^\infty \frac{e_n(x_0;q_2)e_n(1;q_2)}{\lambda_{2,n}}
- \frac{1}{\Gamma(1-\alpha)} \sum_{n=0}^\infty
\frac{e_n(x_0;q_2)e_n(1;q_2)}{\lambda^2_{2,n}} \frac{1}{t^\alpha}
+ O\left(\frac{1}{t^{2\alpha}}\right),\quad \mbox{as }t\to\infty.
\end{aligned}
\end{equation}

Letting $t\rightarrow \infty$ in the above equation \eqref{obser-asymt},
we reach
\begin{equation}\label{obser-phi}
\sum_{n=0}^\infty \frac{e_n(x_0;q_1)e_n(1;q_1)}{\lambda_{1,n}}
= \sum_{n=0}^\infty \frac{e_n(x_0;q_2)e_n(1;q_2)}{\lambda_{2,n}}.
\end{equation}
Taking \eqref{obser-phi} into \eqref{obser-series},
then \eqref{obser-series} can be rephrased as follows
\begin{equation}\label{obser-series'}
\sum_{n=0}^\infty\frac{e_n(x_0;q_1)e_n(1;q_1)}{\lambda_{1,n}}
E_{\alpha,1}(-{\lambda_{1,n}}t^\alpha)
= \sum_{n=0}^\infty\frac{e_n(x_0;q_2)e_n(1;q_2)}E_{\alpha,1}
(-{\lambda_{2,n}}t^\alpha).
\end{equation}
Here we can argue similarly to \cite{Cheng2009uniqueness,Jin2012inverse} to
deduce that the series in \eqref{obser-series'} converge uniformly for
$t\in(0,\infty)$. Therefore, taking the Laplace transform term-wisely
on both sides of \eqref{obser-series'} gives
$$
\sum_{n=0}^\infty\frac{e_n(x_0;q_1)e_n(1;q_1)}{\lambda_{1,n}}
\frac{\zeta^{\alpha-1}}{\zeta^\alpha + \lambda_{1,n}}
= \sum_{n=0}^\infty\frac{e_n(x_0;q_2)e_n(1;q_2)}{\lambda_{2,n}}
\frac{\zeta^{\alpha-1}}{\zeta^\alpha + \lambda_{2,n}},\quad \Re \zeta>0,
$$
from which we further divide both sides of the above equation by
$\zeta^{\alpha-1}$ implies
\begin{equation}\label{obser-lap}
\ \sum_{n=0}^{\infty}\frac{e_n(x_0;q_1)e_n(1;q_1)}{\lambda_{1,n}}
\frac{1}{\zeta + \lambda_{1,n}}
= \ \sum_{n=0}^{\infty}\frac{e_n(x_0;q_2)e_n(1;q_2)}{\lambda_{2,n}}
\frac{1}{\zeta + \lambda_{2,n}}.
\end{equation}
By the asymptotic expansions of the eigenfunctions in \eqref{asymp-phi} and
the eigenvalues $\lambda_n = O(n)$, we can analytically continue both sides of
\eqref{obser-lap} in $\zeta$ when the above series are convergent, so that
\eqref{obser-lap} holds true for $\zeta\in \mathbb{C}\setminus
( \{-\lambda_{1,n}\}_{n=0}^{\infty} \cup \{-\lambda_{2,n}\}_{n=0}^{\infty})$.
Noting $\phi_{j, n}(1) \neq 0$ and $\lambda_{j, n}>0$, $j=1,2$, we can assert that there exists $m(n) \in \mathbb{N} \cup \{0\}$ such that
\begin{equation}\label{eq-eigen}
\lambda_{1,n} = \lambda_{2,m(n)}
\end{equation}
in the case of $\phi_{1,n}(x_{0}) \neq 0$. If not, assume that for any
$n_0\in \mathbb{N} \cup \{0\}$, the inequality $\lambda_{1, n_0} \neq
\lambda_{2,m}$ holds for all $m \in \mathbb{N} \cup \{0\}$.
Then we can choose a circle which includes $-\lambda_{1,n_0}$ and excludes
$\{ -\lambda_{1, k}\}_{k \neq n_0} \cup \{ -\lambda_{2, k}\}_{k \geq 0}$.
Integrating \eqref{obser-lap} along the circle
and applying the Cauchy theorem, we obtain
$$
\frac{2\pi i e_{n_0}(x_0;q_1)e_{n_0}(1;q_1)}{\lambda_{1,n}}=0,
$$
which is impossible by $\phi_{j, n}(1) \neq 0$, $\phi_{1,n}(x_{0}) \neq 0$
and $\lambda_n>0$, $n \in \mathbb N \cup \{0\}$.
Thus we obtain \eqref{eq-eigen} since $n_0$ is arbitrary.
Furthermore, from \eqref{obser-lap} and \eqref{eq-eigen}, by repeating
the above argument, we can see that
\begin{equation}\label{eq-phi12}
e_n(1;q_1)e_n(x_{0};q_1) = e_{m(n)}(1;q_2)e_{m(n)}(x_{0};q_2).
\end{equation}

Next, we will prove Theorem \ref{thm-unique} separately in the following two
cases:

{\bf Case (i)}: $x_0 \geq d$.

By means of the estimates in Lemma \ref{lem-F}, we conclude from Proposition
B.6 in \cite{F2000inverse} that $F \equiv 0$, with which \eqref{eq-F} yields
$$
m_{1,-}(d,s) = m_{2,-}(d,s), \quad \quad s\in \mathbb C.
$$
Therefore, in view of the assumption $q_1 = q_2$ in $[d, 1]$ and Marchenko's
uniqueness theorem (see e.g., Theorem 2.1 in \cite{GW2015}),
we obtain $h_1=h_2$ and $q_1 = q_2$ in $[0,1]$. The proof in Case (i)
is completed.

\vspace{0.6em}

{\bf Case (ii)}: $0 \le x_{0} < \min\{ d, -2d+1\}$ and
$0<d<\frac{1}{2}$.

In this case, we can directly see that $0 < x_0 < \frac{1}{2}$.
According to \eqref{esti-N}, it follows that
$$
 N_\Lambda(s) \geq (1-x_{0}) \frac{s^{\frac{1}{2}}}{\pi}
$$
for sufficiently large $s>0$. Then from Theorem 1.3 in \cite{F2000inverse} and
the assumption $q_{1}=q_2$ on $[d,1]$, we can derive
$h_1=h_2$ and $q_1 = q_2$ in $[0,1]$. The proof of Cases (ii) is completed.
Thus, we complete the proof of Theorem \ref{thm-unique}.
\end{proof}

\begin{proof}[\bf Proof of Theorem 2]
The proof is similar to the proof of Theorem \ref{thm-unique}. In the same way as Theorem \ref{thm-unique} for obtaining \eqref{eq-eigen} and \eqref{eq-phi12}, it can easily be seen from \eqref{ineq-Lambda} that
$$
N_\Lambda (s) \geq A N_{\sigma(L(q_1))}(s) + B \geq  2d N_{\sigma(L(q_1))}(s) + \frac12 - d.
$$
for large enough $s>0$, where $A \geq 2d$ and $B \geq \frac12-d$. Then from Theorem 1.3 in \cite{F2000inverse}, we have $q_1 = q_2$ in $[0, 1]$ and $h_1=h_2$.
Thus the proof of Theorem 2 is completed.
\end{proof}
\section{Conclusions}

In this paper, we investigated an inverse coefficient problem
in a one dimensional time fractional diffusion equation with nonhomogeneous
Robin boundary conditions. Based on the properties of the solution for the
forward problem, Laplace transforms, and inverse spectral theory of the
spatial operator, we proved the simultaneous uniqueness of determining a
spatially varying potential and one Robin coefficient by using the additional
interior measurement.

It should be noted that when the boundary conditions are other cases in
\eqref{eq-gov'}, the corresponding uniqueness results in Theorems 1 and 2
remain valid.

We conclude this article with a possible future topic on the inverse
problem for model \eqref{eq-gov'} from the interior measurement. It is well
known that the fractional-order $\alpha$ is a significant parameter for the
time-fractional diffusion equation. However, the unique determination of the
order $\alpha$ is not obtained in this work. Thus, a new approach needs to be
established, and the simultaneous uniqueness of the fractional-order $\alpha$,
potential $q(x)$, and Robin coefficient $h$ is a future topic.
\section*{Acknowledgments}
The work of first author is supported by the National Natural Science
Foundation of China, (no. 12301534).
The second author thanks National Natural Science Foundation of China (no. 12271277) and Ningbo Youth Leading Talent Project (no. 2024QL045). This work is partly supported by the Open Research Fund of Key Laboratory of Nonlinear Analysis \& Applications (Central China Normal University), Ministry of Education, China.
The work of third author is supported by Grant-in-Aid for Scientific Research
(A) 20H00117 and Grant-in-Aid for Challenging Research (Pioneering) 21K18142
of Japan Society for the Promotion of Science.


\end{document}